\documentclass[12pt]{amsart}
\usepackage{amsmath,amsthm,amssymb,bbm,dsfont}
\usepackage{graphicx}
\usepackage{subfig}
\usepackage{cite,float}

\numberwithin{equation}{section}
\newtheorem{theorem}{Theorem}[section]
\newtheorem{lemma}[theorem]{Lemma}
\newtheorem{result}[theorem]{Result}
\newtheorem{corollary}[theorem]{Corollary}
\newtheorem{proposition}[theorem]{Proposition}
\theoremstyle{remark}
\newtheorem{remark}[theorem]{Remark}

\newtheorem{definition}[theorem]{Definition}

\makeatletter
\@namedef{subjclassname@2010}{%
  \textup{2010} Mathematics Subject Classification}
\makeatother
\begin{document}

\title{$d$-balanced squeezing function }

\author[N. Gupta]{Naveen Gupta}
\address{Department of Mathematics, University of Delhi,
Delhi--110 007, India}
\email{ssguptanaveen@gmail.com}

\author[S. Kumar]{Sanjay Kumar Pant}

\address{Department of Mathematics, Deen Dayal Upadhyaya college, University of Delhi,
Delhi--110 078, India}
\email{skpant@ddu.du.ac.in}

\begin{abstract}
We introduce the notion of $d$-balanced squeezing function 
motivated by the concept of generalized squeezing function given by Rong and Yang.
In this work we study some of its properties and its relation with the Fridman invariant. 
 \end{abstract}
\maketitle

\section{Introduction}
The aim of this article is to extend the notion of generalized squeezing function
for balanced domains 
introduced by Rong and Yang in \cite{gen}. We have extended it to
$d$-balanced domains. Before giving our definition we give, in chronological order,
the notions preceding it.

$\mathbb{B}^n$ denotes unit ball in $\mathbb{C}^n$ and $D\subseteq \mathbb{C}^n$  is used for
bounded domain. The set  of all injective holomorphic 
maps from $D$ to  a domain $\Omega\subseteq \mathbb{C}^n$ is denoted by 
$\mathcal{O}_u(D,\Omega)$.

For $z\in{D}$ 
the squeezing function $S_{D}$ on $D$ is defined as
$$S_{D}(z):=\sup_f\{r:B^n(0,r)\subseteq f(D), f\in{\mathcal{O}_u(D,\mathbb{B}^n)}\},$$ 
where $B^n(0,r)$ denotes ball of radius $r$ centered at the origin.
\smallskip

In our recent article \cite{npoly}, we introduced a definition of squeezing function $T_D$ 
corresponding to polydisk 
$\mathbb{D}^n$ in $\mathbb{C}^n$: 
$$T_{D}(z):=\sup_f\{r:\mathbb{D}^n(0,r)\subseteq f(D), f\in{\mathcal{O}_u(D,\mathbb{D}^n)}\},$$ 
where $\mathbb{D}^n(0,r)$ denotes polydisk of radius $r$, centered at the origin.

In \cite{gen}, Rong and Yang introduced the concept of generalized squeezing function
$S^{\Omega}_D$ for bounded domains $D,\Omega\subseteq \mathbb{C}^n$, where $\Omega$ is a
balanced domain.

Let us quickly recall the notion of balanced domains and Minkowski function. We say that a 
domain $\Omega\subseteq \mathbb{C}^n $ is balanced if  $\lambda z\in{\Omega} $
for each $z\in{\Omega}$ and $|\lambda|\leq 1$. Let $\Omega \subseteq \mathbb{C}^n$ 
be a bounded, balanced, convex
domain. The Minkowski function denoted by $h_{\Omega}$  on 
$\mathbb{C}^n$ is defined as 
$$h_{\Omega}(z):=\inf \{t>0:z/t\in{\Omega}\}.$$
For $0<r\leq 1$, let $\Omega(r):=\{z\in{\mathbb{C}^n:h_{\Omega}(z)<r\}}$. 
It can be seen easily that $\Omega(1)=\Omega$. For 
a bounded domain $D\subseteq \mathbb{C}^n$ and a bounded, balanced, convex domain 
$\Omega\subseteq \mathbb{C}^n$, Rong and Yang introduced the notion of generalized squeezing
function $S_D^{\Omega}$ on $D$ as
\begin{equation*}\label{eqn:gensq}
	S^{\Omega}_D(z):=\sup \{r:\Omega(r)\subseteq f(D), f\in{\mathcal{O}_u(D,\Omega)}, f(z)=0\}.
\end{equation*} 

It follows from the definition that $S^{\Omega}_D$
is biholomorphic invariant and that its values lie in semi-open interval $(0,1].$ 
As for squeezing function in general, a bounded domain $D$ is holomorphic 
homogeneous regular if its
generalized squeezing function $S^{\Omega}_D$ has a positive lower bound.

Motivated by the notion of balanced domains, Nikolov in his work \cite{nikolov-d}
gave the definition of $d$-balanced(quasi balanced) domains: Let 
$d=(d_1,d_2,\ldots,d_n)\in{\mathbb{Z}_n^{+}},n\geq 2$, a domain $\Omega\subseteq \mathbb{C}^n$
is said to be $d$-balanced if for each $z=(z_1,z_2,\ldots, z_n)\in \Omega$
and $\lambda \in{\overline{\mathbb{D}}}$, $\left(\lambda^{d_1}z_1,\lambda^{d_2}z_2,\ldots, 
\lambda^{d_n}z_n\right) \in \Omega,$ where $\mathbb{D}$ denotes unit disk in $\mathbb{C}$.
Note that balanced domains are simply $(1,1,\ldots, 1)$-balanced.

For a $d$-balanced domain $\Omega$, there is a natural analogue of Minkowski function called
the $d$-Minkowski function on $\mathbb{C}^n$,
denoted by $h_{d,\Omega}$ and is defined as 
$$h_{d,\Omega}(z):=\inf \{t>0:\left(\frac{z_1}{t^{d_1}},\frac{z_2}{t^{d_2}},\ldots,
\frac{z_n}{t^{d_n}}\right)\in{\Omega}\}.$$
For each $0<r\leq 1$, we fix 
$\Omega^d(r):=\{z\in{\mathbb{C}^n:h_{d,\Omega}(z)<r\}}$.
It is easy to observe that for a bounded $d$-balanced domain $\Omega$, $\Omega^d(1)=\Omega\,$(Remark 
\ref{rem:basicminko}). Finally we are in a position to introduce the definition
of  our  $d$-balanced squeezing function.

\begin{definition} 
	For a bounded domain $D\subseteq \mathbb{C}^n$, and a bounded, convex, 
	$d=(d_1,d_2,\ldots, d_n)$-balanced domain $\Omega$, 
	\emph{ $d$-balanced squeezing function corresponding to $\Omega$} of the domain $D$,  
	denoted by $S_{d,D}^\Omega$(also called the\emph{ $d$-balanced squeezing function} for brevity)
	is given by:
	\begin{equation*}\label{eqn:gensq}
		S_{d,D}^\Omega(z):=\sup \{r:\Omega^d(r)\subseteq f(D), f\in{\mathcal{O}_u(D,\Omega)}, f(z)=0\}.
	\end{equation*}
\end{definition}
For notational convenience, we denote the squeezing function  
$S_{d,D}^\Omega$ by $S^d$ unless otherwise stated.
It is easy to see that $S^d$ is biholomorphic invariant and its values lie in the semi-open
interval $(0,1].$ As with generalized squeezing function here also $D$ is holomorphic homogeneous d-regular if its
squeezing function $S^d$ has a positive lower bound.

Recall that a domain is said to be homogeneous if its group of automorphisms 
acts transitively on it. Let  $D\subseteq \mathbb{C}^n$ be bounded and 
$\Omega\subseteq \mathbb{C}^n$
be bounded, homogeneous. Fridman invariant on $D$, denoted by $g_{D}^d$,
is defined as $$g_{D}^d(a):=\inf\{1/r :B_{D}^d(a,r)\subseteq f(\Omega),
f\in{\mathcal{O}_u(\Omega,D)}\},$$
where $d_{D}$ is the Carathéodory(or Kobayashi) pseudodistance  $c_{D}(\mbox{or} \ k_{D})$ on $D$ and 
$B_{D}^d(a,r)$ is the $c_{D}(\mbox{or}\ k_D)$  ball centered at $a$ of radius $r>0.$ 
For comparison purpose, we consider $h_{D}^d$, defined as
$$h_{D}^d(a):=\sup\{\tanh r :B_{D}^d(a,r)\subseteq f(\Omega), 
f\in{\mathcal{O}_u(\Omega,D)}\}.$$

This work is devoted to general properties of $S_{d,D}^{\Omega}$ and $h_D^d$, and some connections between them.
The question arises: given all the information that the(standard) squeezing function already provides, what 
could be the utitlity of $S_{d,D}^{\Omega}$?  We point to a host of results where, given $D$ as above and a point $p\in
\partial D$ around which $\partial D$ is smooth, if $\lim_{z\to p}S_D(z)=1$ then, based on additional geometric information
about $\partial D$ around $p$, it is inferred that $\partial D$ is strongly Levi-pseudoconvex at $p$: see \cite{h-extendible} and 
the references therin. Now, $S_{d,D}^{\Omega}$ would play an analogous role in understanding alternative Levi geometries. 
By ``understanding alternative Levi geometries" we mean the following type of problem: with $\Omega$ fixed and 
$d\neq (1,1,\ldots, 1)$ and with some additional geometric information about $\partial D$ around $p$, studying what 
the inertia of the Levi-form of $\partial D$ must be in terms of the non-$1$ entries  of $d$ if 
$\lim_{z\to \partial D} S_{d,D}^{\Omega}(z)=1.$(Note how $d=(d_1,d_2,\ldots, d_k)$ is vital to such problems.)
Some of the properties of $S_{d,D}^{\Omega}$ exposed below would be needed
for the latter class of problems. Now, what sort of geometric information about $\partial D$ around $p$ would be needed, 
one may ask. See \cite[~Section 8]{bharali2021} for a range of natural conditions and for a general discussion on what ingredients the solutions of the
latter problems need.

{\it Layout of the paper}: In the second section we show that product of holomorphic 
homogeneous regular domain is
holomorphic homogeneous regular. An inequality related to $d$-Minkowski function
is the main feature of section three.  The fourth and final sections
contain the usual results about squeezing functions for  $d$-balanced domains and
Fridman invariant. We would like to point out specifically the results concerning Fridman
invariant and the Lemma \ref{lem:bound} achieved by tweaking a result of Bharali [1]. This
lemma was mainstay for our continuity result 4.9 in the fourth section. We would
also like to mention the continuity result regarding the construction of a particular
function $g$.

%Finally we would like to point out that Lemma \ref{lem:bound} is a 
%crucial factor in the proof of our continuity result in section four. 
%\ref{thm:dcts}.
% out the results concerning Fridman
% invariant as well as Lemma \ref{lem:bound} whose proof is achieved by tweaking the proof of 
% a theorem of Bharali 
% \cite{bharali}. This lemma was mainstay for our continuity result \ref{thm:dcts} in the fourth section. 
% We would also like to  mention the continuity result regarding the construction
% of a particular function $g$.
\section[]{On generalized squeezing function}
We begin with the following observation:
\begin{lemma}\label{lem:hom}
	For $0<r<1, \ \Omega(r)=r\, \Omega(1).$
\end{lemma}
\begin{proof}
	Let $z\in{\Omega(r)},$ thus $h_{\Omega}(z)<r$. This, using homogeneity of $h_{\Omega}$
	\cite[~Remark 2.2.1(a)]{pflug}, gives us $h_{\Omega}\left(\frac 1 r z\right)<1.$	Thus
	$\frac 1 r z\in{\Omega(1)}$. Therefore $z=r\left( z/r\right)\in{r\, \Omega(1)}$.
	
	Let $z\in r\, {\Omega(1)}$, then $z=ra\,$ for some $a\in {\Omega(1)}.$ Now using 
	homogeneity of 
	$h_{\Omega}$, we get $h_{\Omega}(r a)<r$, which clearly shows that $ra=z\in{\Omega(r)}.$
	%On the other hand for $z=ra\in{r\, \Omega(1)},\  \mbox{where }\ a\in{\Omega(1)}$, now using 
	%homogenity of 
	%$h_{\Omega}$, we get $h_{\Omega}(r a)<r$, which further gives $ra=z\in{\Omega(r)}.$
\end{proof}
\begin{remark}\label{rem:hom} Let $\Omega_1\subseteq \mathbb{C}^{n_1}$ and $\Omega_2
	\subseteq \mathbb{C}^{n_2}$ be balanced domains
	and let $\Omega=\Omega_1\times \Omega_2$. We know that product of balanced domains is balanced, then
	by the above lemma one can see that for any $r>0$, 
	$\Omega(r)=\Omega_1(r)\times \Omega_2(r).$ 
	%\begin{align*}
	%	\Omega(r)=r\, \Omega(1)&=r\, \Omega\\
	%	&=r\, (\Omega_1 \times \Omega_2)\\
	%	&=r\, \Omega_1\times r\, \Omega_2\\
	%	&=\Omega_1(r)\times \Omega_2(r).
	%	\end{align*}
\end{remark}
Remark \ref{rem:hom}  can also be deduced from 
\cite[~Remark 2.2.1]{pflug}.
On the lines of our result on product domain\cite{npoly}, we deduce a similar result
about generalized squeezing function on product domain.

\begin{proposition}\label{prp:lowerbound} If we consider
\begin{itemize}
	\item $\Omega_i\subseteq \mathbb{C}^{n_i},\ 
	(i=1,2,\ldots,k)$  bounded, convex and balanced domains.
	\item $\Omega=\Omega_1\times \Omega_2\times \ldots \times \Omega_k\subseteq \mathbb{C}^n,$ 
	where $n=n_1+n_2+\ldots +n_k$.
	\item $D_i\subseteq \mathbb{C}^{n_i},\ (i=1,2,\ldots,k)$ bounded domains and
	$D=D_1\times D_2\times \ldots \times D_k$.
\end{itemize}
Then  for $a=(a_1,a_2,\ldots,a_k)\in{D}$, 
\begin{equation}\label{eqn:lowerbound}
	S_{D}^{\Omega}(a)\geq\min_{1\leq i\leq k}  S_{D_i}^{\Omega_i}(a_i). 
\end{equation}

% $\Omega_i\subseteq \mathbb{C}^{n_i},\ 
%i=1,2,\ldots,k$ be bounded, convex and balanced domains. Let 
%$D_i\subseteq \mathbb{C}^{n_i},\ i=1,2,\ldots,k$ be bounded domains.
%Let $\Omega=\Omega_1\times \Omega_2\times \ldots \times \Omega_k\subseteq \mathbb{C}^n,$ 
%where $n=n_1+n_2+\ldots +n_k$ and $D=D_1\times D_2\times \ldots \times D_k$.
%Then for $a=(a_1,a_2,\ldots,a_k)\in{D}$, 
%\begin{equation}\label{eqn:lowerbound}
%S_{D}^{\Omega}(a)\geq\min_{1\leq i\leq k}  S_{D_i}^{\Omega_i}(a_i). 
%\end{equation}
\end{proposition}
\begin{proof}
By \cite[~Theorem 3.4]{gen}, for each $a_i\in{D_i}$, there exists an extremal map. That is
for each $i,\ 1\leq i\leq k$,
there exist injective holomorphic map
$f_i:D_i\to \Omega_i$ with $f_i(a_i)=0$ such that 
\begin{equation}\label{eqn:product}
	\Omega_i\left(S_{D_i}^{\Omega_i}(a_i)\right)\subseteq f_i
	(\Omega_i)\subseteq \Omega,\ i=1,2,\ldots, k.
\end{equation}
Consider the map $f:D \to \Omega$  defined as
$$f(z_1,z_2,\ldots, z_k):=\left(f_1(z_1),f_2(z_2),\ldots,f_k(z_k)\right).$$
Clearly, $f$ is an injective holomorphic map with $f(a)=0$. Let $r=\min_{1\leq i\leq k}  
S_{D_i}^{\Omega_i}(a_i)$. It follows from Remark \ref{rem:hom} that 
$\Omega(r)=\Omega_1(r)\times \Omega_2(r)\times \ldots \times \Omega_k(r)$. 
Let $w=(w_1,w_2,\ldots, w_k) \in \Omega(r)= \Omega_{1}(r)\times \Omega_{2}(r)\times \ldots
\times \Omega_{k}(r). $ By
Equation \ref{eqn:product}, there exists $b_i\in{D_i},$ such that $f_i(b_i)=w_i,\ i=1,2,
\ldots, k,$ since
$\Omega_{i}(0,r)\subseteq \Omega_{i}(S_{D_i}^{\Omega_i}(a_i)) $ for each $i$.
Thus $w=f(b_1,b_2,\ldots, b_k)$ and as $w$ was arbitrarily chosen, we conclude 
$\Omega(r)\subseteq f(D)$. 
Thus it follows from the definition that $S_{D}^{\Omega}(a)\geq\min_{1\leq i\leq k} 
S_{D_i}^{\Omega_i}(a_i).$
\end{proof}
A trivial consequence of the  above inequality (\ref{eqn:lowerbound}) is
that the product of holomorphic homogeneous regular
domains is holomorphic homogeneous regular.

%As a result we have the following corollary.
%\begin{corollary}
%	Product of holomorphic homogeneous regular
%	 domains is holomorphic homogeneous regular.
%\end{corollary}
\section{Few results on $d$-minkowski function}

\begin{remark}\label{rem:basicminko}
For a $d$-balanced domain $\Omega\subseteq \mathbb{C}^n$, the following holds:
( see\cite[~Remark 2.2.14]{pflug})
\begin{enumerate}
	\item $\Omega =\{z\in{\mathbb{C}^n}:h_{d,\Omega}(z)<1\}$.
	\item $h_{d,\Omega}\left(\lambda^{d_1}z_1,\lambda^{d_2}z_2,\ldots, \lambda^{d_n}
	z_n\right)=|\lambda |h_{d,\Omega}(z)$ for each $z=(z_1,z_2,\ldots, z_n)\in{\mathbb{C}^n}$
	and $\lambda\in{\mathbb{C}}.$
	\item $h_{d,\Omega}$ is upper semicontinuous.
\end{enumerate}
\end{remark}

Minkowski function $h_\Omega$ for a bounded, balanced, convex domain $\Omega\subseteq \mathbb{C}^n$ is 
a $\mathbb{C}$-norm \cite[~Lemma 3.3]{gen}, in particular it satisfies triangle inequality. 
In this direction, for a $d$-balanced,
convex domain $\Omega\subseteq \mathbb{C}^n$,
we have the following proposition.

\begin{proposition}\label{prop:triangle}
Let $\Omega\subseteq \mathbb{C}^n$ be a $d$-balanced, convex domain. Then for 
$z,w\in\mathbb{C}^n,\ \alpha\in{[0,1]} $, $$h_{d,\Omega}(\alpha z+(1-\alpha )w)
\leq h_{d,\Omega}(z)+h_{d,\Omega}(w). $$ 
\end{proposition}
\begin{proof}
Let $\epsilon >0$ be arbitrary and $a=h_{d,\Omega}(z)+\epsilon/2,\ b=h_{d,\Omega}(w)+
\epsilon/2$. Then there exists $t,s>0$ with $t<a, \ s<b$ such that $\left(\frac{z_1}{t^{d_1}},
\frac{z_2}{t^{d_2}},\ldots, \frac{z_n}{t^{d_n}}\right)\in{\Omega}$ and 
$\left(\frac{w_1}{s^{d_1}},\frac{w_2}{s^{d_2}},\ldots, \frac{w_n}{s^{d_n}}\right)\in{\Omega}$.
Since $t/a<1,\ s/b<1$ and $\Omega$ is $d$-balanced, we get that $\left(\frac{z_1}{a^{d_1}},
\frac{z_2}{a^{d_2}},\ldots, \frac{z_n}{a^{d_n}}\right)\in{\Omega}$ and 
$\left(\frac{w_1}{b^{d_1}},\frac{w_2}{b^{d_2}},\ldots, \frac{w_n}{b^{d_n}}\right)
\in{\Omega}$. Let $c=\max(a,b)$, then we get $\left(\frac{z_1}{c^{d_1}}, \frac{z_2}{c^{d_2}},
\ldots, \frac{z_n}{c^{d_n}}\right)\in{\Omega}$ and $\left(\frac{w_1}{c^{d_1}},
\frac{w_2}{c^{d_2}},\ldots, \frac{w_n}{c^{d_n}}\right)\in{\Omega}$. Using convexity, 
we get $\left(\frac{\alpha z_1+(1-\alpha )w_1}{c^{d_1}}, \frac{\alpha z_2+(1-\alpha )
	w_2}{c^{d_2}},\ldots, \frac{\alpha z_n+(1-\alpha )w_n}{c^{d_n}}\right)\in{\Omega}=
\Omega^d(1).$ Therefore we get $$h_{d,\Omega}\left(\frac{\alpha z_1+(1-\alpha )w_1}{c^{d_1}},
\frac{\alpha z_2+(1-\alpha )w_2}{c^{d_2}},\ldots,
\frac{\alpha z_n+(1-\alpha )w_n}{c^{d_n}}\right)<1,$$  which upon using Remark 
\ref{rem:basicminko}(2) gives us $h_{d,\Omega}(\alpha z+(1-\alpha )w)<c$. Noting that
$c=\frac 1 2 \left(a+b+|a-b|\right)$, we get $h_{d,\Omega}(\alpha z+(1-\alpha )w)
<h_{d,\Omega}(z)+h_{d,\Omega}(w)+\epsilon.$ Since $\epsilon $ was arbitrary, we conclude
that $h_{d,\Omega}(\alpha z+(1-\alpha )w)\leq h_{d,\Omega}(z)+h_{d,\Omega}(w).$
\end{proof}

%\begin{lemma}\label{lem:symmetric}
%Let $\Omega\subseteq \mathbb{C}^n$ be a $d=(d_1,d_2,\ldots, d_n)$-balanced domain, where 
%$d_i$'s are all odd, then $\Omega$ is symmetric, that is $-z\in{\Omega}$ whenever
% $z\in{\Omega}.$  Moreover, $h_{d,\Omega}(z)=h_{d,\Omega}(-z).$
%\end{lemma}
%\begin{proof}
%	For $z\in{\Omega}$ and $\lambda=-1$, we get $-z=((-1)^{d_1}z_1,(-1)^{d_2}z_2,\ldots, 
%	(-1)^{d_n}z_n)\in{\Omega}.$ Also
%	\begin{align*}
%	h_{d,\Omega}(z)&=h_{d,\Omega}(-(-z))\\
%	&=h_{d,\Omega}((-1)^{d_1}-z_1,(-1)^{d_2}-z_2,\ldots, (-1)^{d_n}-z_n)\\
%	&=|-1|h_{d,\Omega}(-z)=h_{d,\Omega}(-z).
%	\end{align*}	
%	We are using Remark \ref{rem:basicminko}(2) in the last step.
%\end{proof}

\section{$d$-balanced squeezing function}

We first recall few results that we will be using in this section.
Note that \cite[~Theorem 1]{lempert-classic}, \cite[~Theorem 1.3]{lempert}
and the Remark 1.6 therein yields the following.
\begin{result}\label{res:lempert}
For a convex domain $\Omega\subseteq \mathbb{C}^n,\ c_{\Omega}
=k_{\Omega}=\tilde{k}_{\Omega},$
where $\tilde{k}_{\Omega}$ denotes the Lempert function on $\Omega$.
\end{result}
Combining Result \ref{res:lempert} with \cite[~Theorem 1.6]{bharali}, we get the following.

\begin{result}\label{res:bharali}
For a bounded, convex, $d=(d_1,d_2,\ldots,d_n)$-balanced domain $\Omega\subseteq \mathbb{C}^n$, 
$$\tanh^{-1}h_{d,\Omega}(z)^L\leq c_{\Omega}(0,z)=k_{\Omega}(0,z)\leq \tanh^{-1}h_{d,\Omega}(z),$$
where $L=\max_{1\leq i\leq n}d_i.$
\end{result}

\begin{result}[{\cite[Theorem~2.2]{2012}}]\label{res:injectivity} 
Let $D\subseteq \mathbb{C}^n$ be a bounded domain and $z\in D$. Let $\{f_i\}$
be a sequence of injective holomorphic
maps, $f_i:D\to \mathbb{C}^n$, with $f_i(z)=0$ for all $i.$
Suppose that $f_i\to f,$ uniformly on compact subsets of $D$,
where $f:D\to \mathbb{C}^n$. If there exists a neighborhood
$U$ of $0$ such that $U\subseteq f_i(D)$ for all $i$, then $f$ is 
injective.
\end{result}

\begin{result}[{\cite[Lemma~2.4]{llyod}}]\label{res:lemrouche}
Let $D\subseteq \mathbb{C}^n$ be a bounded domain and $f,\, g:D\to \mathbb{C}^n$
be holomorphic such that 
$$\|f(z)\|< \|g(z)\|\ ,z\in \partial D. $$ Then $f$ and $f+g$ have the same number of zeroes
in $D$, counted according to multiplicities.
\end{result}
For a bounded domain $D$ and a bounded, convex, $d$-balanced domain $\Omega$, an injective
holomorphic map $f:D\to \Omega$
with $f(z)=0$ is said to an be an extremal map at $z\in D$, if 
$\Omega^d(S^{d}(z))\subseteq f(D).$ Before proving the existence
of extremal maps for  squeezing function $S^d$ we prove the following useful lemma.

\begin{lemma}\label{lem:dextremal}
Let $\Omega\subseteq \mathbb{C}^n$ be $d$-balanced domain and $r_k\to r,\ 0<r_k,r<1$ be
such that  for every $k$, $\Omega^d(r_k)\subseteq A$, where $A$ is some subset of
$ \mathbb{C}^n$. Then $\Omega^d(r) \subseteq A.$ 
\end{lemma}
\begin{proof}
Let $z=(z_1,z_2,\ldots,z_n)\in{\Omega^d(r)}$, that is $h_{d,\Omega}(z)<r$. 
Thus $r$ is not a lower bound for 
$B,$ where $B=\{t>0:\left(\frac{z_1}{t^{d_1}},\frac{z_2}{t^{d_2}},\ldots,
\frac{z_n}{t^{d_n}}\right)\in{\Omega}\}$. Therefore there is $t_0>0$ 
such that $t_0\in B$ with $t_0<r.$

For $\epsilon =r-t_0$, choose $N\in{\mathbb{N}}$ 
such that $r_N>t_0.$ Thus $h_{d,\Omega}(z)=\inf B\leq t_0<r_N$. This gives us
$z\in \Omega^d(r_N)\subseteq A$. Thus we get $\Omega^d(r)\subseteq A.$
\end{proof}

\begin{theorem}\label{thm:dextremal} 
Let $\Omega \subseteq \mathbb{C}^n$ be bounded, convex and $d$-balanced domain and 
$D\subseteq \mathbb{C}^n$ be a bounded domain. Then for $a\in D$, 
there exists an injective holomorphic map $f:D\to \Omega$ with $f(a)=0$
such that $\Omega^d(S^d(a))\subseteq f(D).$
\end{theorem}
\begin{proof}
Let $a\in{D}$ and $r=S^d(a)$. Let $r_i$ be a sequence of increasing numbers converging 
to $r$ and let $f_i:D\to \Omega$ be injective holomorphic map with $f_i(a)=0$ such that
$$\Omega^d(r_i)\subseteq f_i(D),\ \mbox{for each}\ i.$$

Since each $f_i(D)\subseteq \Omega$, therefore the sequence $f_i$ is locally bounded
and hence normal. Thus by Montel's theorem, there exists a subsequence $f_{i_k}$ of $f_i$
such that $f_{i_k}\to f$ uniformly on compact subsets of $D$. 
Clearly, $f:D\to \overline{\Omega}$ is a holomorphic map with 
$f(a)=0$. As $r_i$ is an increasing sequence therefore
$\Omega^d(r_1)\subseteq f_i(D)$ for every $i.$ As we know that $h_{d,\Omega}$ is upper 
semicontinuous  therefore $\Omega^d(r_1)$ is open. Now using 
Result \ref{res:injectivity}, we get that $f$ is an open map. We know that $\Omega$ being convex
is a fat domain \cite[~Remark 1.4.1(h)]{rein},  which results in $f:D\to 
\mbox{int}\ \overline{\Omega}=\Omega$. 
Now we show that $\Omega^d(S^{d}(a))\subseteq f(D)$. For this, it suffices to prove
that
$\Omega^d(S^{d}(r_j))\subseteq f(D)$ for each fixed $j.$ Finally  Lemma \ref{lem:dextremal}
concludes the proof.

Note that for each $i>j\, (\mbox{j is fixed}), \  \Omega^d(r_j)\subseteq \Omega^d(r_j)\subseteq f_i(D).$ 
Now  consider map $g_i:\Omega^d(r_j)\to D$  defined as 
$g_i=f_i^{-1}|_{\Omega^d(r_j)}$ for each $i>j $, then $f_{i_k}\circ g_{i_k}=
\mathbbm{Id}_{\Omega^d(r_j)}$ for $i_k>j$. By Montel's theorem, sequence $g_{i_k}$
has a subsequence, naming it again  $g_{i_k}$, uniformly converging to a function 
$g:\Omega^d(r_j)\to \mathbb{C}^n$ on compact subsests of $\Omega^d(r_j)$.
It can be seen easily that $g$ is locally biholomorphic.

Clearly, $g:\Omega^d(r_j)\to \overline{D}$. We claim that $g:\Omega^d(r_j)\to D$. 
For this, first note that $g$ is defined on some neighborhood of the closure
$\overline{\Omega^d(r_j)}$. Suppose there is  $\zeta 
\in{g(\Omega^d(r_j))}$ such that $\zeta \notin D$. Let $\tilde {g_{i_k}}(z)=
g_{i_k}(z)-\zeta$ and $\tilde {g}(z)=g(z)-\zeta$ for $z\in \Omega^d(r_j) .$
Since $g_{i_k}\left(\Omega^d(r_j)\right)\subseteq  D$, therefore $\tilde{g_{i_k}}$ has no
zero in $\Omega^d(r_j)$ and $\tilde{g}$ has a zero in $\Omega^d(r_j)$.
Let $z_0\in \Omega^d(r_j)$ be such that $g(z_0)=\zeta$, that is, $\tilde{g}(z_0)=0$.
Since $g$ is locally biholomorphism, there is some $\delta>0$ such that $z_0$ is
the unique zero of $\tilde{g}$ on $\overline{B^n(z_0,\delta)}.$ Take $\epsilon 
=\inf \{|\tilde{g}(z)|:\partial B^n(z_0,\delta)\}$ and note that $\epsilon >0$.
Now using convergence of $\tilde{g_{i_k}}$ for this $\epsilon$ and then using Result 
\ref{res:lemrouche}, we get that $\tilde{g_{i_k}}$ has a zero in $B^n(z_0,\delta)$
for sufficiently large $k$. It is a contradiction therefore 
$g\left( \Omega^d(r_j)\right)\subseteq D$ for each $j$.

\end{proof} 
The following corollary is an immediate consequence.
\begin{corollary}\label{cor:dextremal}
If $S^d(z)=1$ for some $z\in D$, then $D$ is biholomorphically equivalent to $\Omega$.	
\end{corollary}
Note that for a bounded, convex $d$-balanced domain $\Omega$, $a\Omega$ is also bounded, 
convex $d$-balanced for every $a\in{\mathbb{R}}$. We need the following lemma to prove 
continuity of $S^d$, whose proof
follows on the same lines as the proof of \cite[~Theorem 1.6]{bharali}. We include its proof here for the
sake of completion.
\begin{lemma}\label{lem:bound}
For a bounded, convex, $d$-balanced domain $\Omega$,
$$h_{d,\Omega}(z)\leq B_{a\Omega} \left(\tanh \tilde{k}_{a\Omega}(0,z)\right)^{1/L},$$
where $a\in \mathbb{R}$ and $B_{a\Omega}>0$ is such that $h_{d,\Omega}(z)\leq B_{a\Omega}$
for every $z\in{a\Omega} ($existence of the bound $B_{a\Omega}$ is easy to check.$)$ 

%	$($It can be noted that such bound exist because for some positive real number $R$  the restriction of $h_{d,\Omega}$
%	 to the $\overline{B^n(0,R)}$, where $a\Omega\subseteq B^n(0,R)$. Now upper semicontinuity of $h_{d,\Omega}$
%	  does the work.$)$

%	 $a\Omega$ is bounded and therefore $a\Omega\subseteq 
%	B^n(0,R)$ for some $R>0$, restricting $h_{d,\Omega}$ to $\overline{B^n(0,R)}$ and 
%	observing that $h_{d,\Omega}$ is upper semicontinuous, we obtain such a bound.)
\end{lemma}
\begin{proof}
Observe that if $a=0$, the conclusion is obvious, therefore we assume that $a\neq 0$.	
Let $L=\max_{1\leq i\leq n}d_i$. For any $\zeta \in{\mathbb{D}^*}$, where $\mathbb{D}^*$ is 
punctured unit disc in $\mathbb{C}$, denote by $\tau_1(\zeta),\ldots, \tau_L(\zeta)$ 
distinct $L$th roots of $\zeta$. Let $z\in{a\Omega}$ and $\phi:\mathbb D\to a\Omega$ be 
holomorphic
such that 
$\phi(0)=0$ and $\phi(\sigma)=z$ for some $\sigma \in \mathbb{D}.$ Since $\phi(0)=0$, 
therefore $\phi(\zeta)=\left(\zeta \phi_1(\zeta),\ldots, \zeta \phi_n(\zeta)\right)$ for 
$\zeta\in {\mathbb{D}}$, where $\phi_i:\mathbb{D}\to \mathbb{C}$.

Consider function $U$ defined on $\mathbb{D}^*$ as 
$$U(\zeta):=\sum_{j=1}^L h_{d,\Omega}\left(\tau_j (\zeta)^{L-d_1}\phi_1(\zeta),\ldots,
\tau_j(\zeta)^{L-d_n}\phi_n(\zeta)\right).$$
Following verbatim the argument in the proof of \cite[~Theorem 1.6]{bharali}, we obtain that
$U$ extends to a subharmonic function on $\mathbb{D}$ and for each $r\in{(0,1)}$,
$$r^{1/L}U(\zeta)=L h_{d,\Omega}\circ \phi(\zeta)<L B_{a\Omega}\ \mbox{for every}\
\zeta \ \mbox{with }\  |\zeta|=r.$$
This implies that $U(\zeta)\leq LB_{a\Omega}$ for every $\zeta\in \mathbb{D}$. Therefore
$$Lh_{d,\Omega}(z)=Lh_{d,\Omega}\circ \phi(\sigma)=|\sigma|^{1/L}U(\sigma)\leq LB_{a\Omega} 
|\sigma|^{1/L}.$$
So we get $$\tilde{k}_{a\Omega}(0,z)\geq \rho\left(0,\frac{1}{B_{a\Omega}^L}
h_{d,\Omega}(z)^L\right)=\tanh^{-1}\left(\frac{1}{B_{a\Omega}^L}h_{d,\Omega}(z)^L\right),$$
where $\rho$ denotes Poincar\'e distance on $\mathbb{D}$ and this completes the proof of the 
lemma. 
\end{proof}
\begin{theorem}\label{thm:dcts}
Let $\Omega\subseteq \mathbb{C}^n$ be bounded, homogeneous, $d$-balanced domain and 
$D\subseteq \mathbb{C}^n$ be bounded. Then squeezing function, $S^{d}$ is
continuous.
\end{theorem}
\begin{proof}
Let $z_1, z_2\in D$. Using Theorem \ref{thm:dextremal} for $z_1$, there exists an 
injective holomorphic map $f:D\to \Omega$ with $f(z_1)=0$ such that 
\begin{equation}\label{eqn:thm:dcts}
	\Omega^d(S^d(z_1))
	\subseteq f(D).
\end{equation}
Set $\mathcal{K}=(\tanh k_D(z_1,z_2))^{ 1/L},$ and $k=B_{-\Omega}\mathcal{K}$ 
where $L=\max_{1\leq i\leq n}d_i$ and 
$B_{-\Omega}$ is as in Lemma \ref{lem:bound} for $a=-1$.
If $h_{d,\Omega}(2f(z_2))\geq S^d(z_1)$, then obviously
$$S^d(z_2)>0\geq \frac{S^d(z_1)-h_{d,\Omega}(2f(z_2))}{1+k}.$$

Let us consider the case when $h_{d,\Omega}(2f(z_2))<S^d(z_1)$. Consider $g:D\to
\mathbb{C}^n$ defined as
$$g(z):=\left(\frac{f_1(z)-f_1(z_2)}{2(1+k)^{d_1}},\frac{f_2(z)-f_2(z_2)}
{2(1+k)^{d_2}},\ldots, \frac{f_n(z)-f_n(z_2)}{2(1+k)^{d_n}}\right).$$
Notice that $g$ is injective holomorphic with $g(z_2)=0$. We first claim that
$g(D)\subseteq \Omega$. For $z\in D$ we show that $g(z)\in{\Omega^d(1)=\Omega}$.
By using Remark \ref{rem:basicminko}[1 and 2] and Proposition \ref{prop:triangle} for $\alpha =1/2$,
we have
\begin{align*}
	h_{d,\Omega}(g(z))&=h_{d,\Omega}\left(\frac{f_1(z)-f_1(z_2)}{2(1+k)^{d_1}},
	\frac{f_2(z)-f_2(z_2)}{2(1+k)^{d_2}},\ldots, \frac{f_n(z)-f_n(z_2)}{2(1+k)^{d_n}}\right)\\
	&=\frac{1}{1+k}h_{d,\Omega}\left(\frac{f(z)-f(z_2)}{2}\right)\\
	&\leq\frac{1}{1+k}\left(h_{d,\Omega}(f(z))+h_{d,\Omega}(-f(z_2)\right)\\
	&<\frac{1}{1+k}\left(1+h_{d,\Omega}(-f(z_2)\right).
	%	&\leq \frac{1}{1+k}\left(1+k\right)=1.
\end{align*}
Now using Lemma \ref{lem:bound}(for $a=-1$) we have 

\begin{align*}
	h_{d,\Omega}(g(z))
	&<\frac{1}{1+k}\left(1+h_{d,\Omega}(-f(z_2)\right)\\
	&<\frac{1}{1+k}\left(1+B_{-\Omega}\tanh k_{-\Omega}(0,-f(z_2))^{1/L}\right)\\
	%&\leq B_{-\Omega}\left(\tanh k_{-\Omega}(0,-f(z_2))\right)^{1/L} \hspace{1.6in}
	%	\mbox{(using Lemma \ref{lem:bound})}\\
	&\leq\frac{1}{1+k}\left(1+ B_{-\Omega}\tanh k_{h(D)}(h(z_1),h(z_2))^{1/L}\right)\\
	&=\frac{1}{1+k}\left(1+ B_{-\Omega}(\tanh k_{D}(z_1,z_2))^{1/L}\right)\\
	&=1,
\end{align*}
%	\begin{align*}
	%h_{d,\Omega}(-f(z_2))
	%    &\leq B_{-\Omega}\left(\tanh k_{-\Omega}(0,-f(z_2))\right)^{1/L} \hspace{1.6in}
	%    \mbox{(using Lemma \ref{lem:bound})}\\
	%	&\leq B_{-\Omega}\left(\tanh k_{h(D)}(h(z_1),h(z_2))\right)^{1/L}\\
	%	&= B_{-\Omega}(\tanh k_{D}(z_1,z_2))^{1/L}\\
	%&=k,
	%\end{align*}
	where $h:D\to -\Omega$ is defined as $h(z)=-f(z)$.
	%Therefore $g:D\to \Omega$.
	Next we claim that
	$$\Omega ^d\left(\frac{S^d(z_1)-h_{d,\Omega}(2f(z_2))}{(1+k)}\right)\subseteq g(D).$$
	Let us take $w\in{\Omega ^d\left(\frac{S^d(z_1)-h_{d,\Omega}(2f(z_2))}{(1+k)}\right)}.$
	Therefore $h_{d,\Omega}(w)<\frac{S^d(z_1)-h_{d,\Omega}(2f(z_2))}{(1+k)}$, which upon using
	Remark \ref{rem:basicminko}(2) and Proposition \ref{prop:triangle}(for $\alpha =1/2$) 
	yields $$h_{d,\Omega}\left(2w_1(1+k)^{d_1}-f_1(z_2), \ldots, 2w_n(1+k)^{d_n}-f_n(z_2)\right)
	<S^d(z_1).$$
	This further gives us $$(2w_1(1+k)^{d_1}-f_1(z_2), \ldots, 2w_n(1+k)^{d_n}-f_n(z_2))
	\in{\Omega^d(s^d(z_1))}\subseteq f(D).$$
	Therefore  
	$(2w_1(1+k)^{d_1}-f_1(z_2), \ldots, 2w_n(1+k)^{d_n}-f_n(z_2))=(f_1(a),\ldots, f_n(a))$ for some 
	$a\in{D}$. Thus we get $$w=\left(\frac{f_1(a)-f_1(z_2)}{2(1+k)^{d_1}},\frac{f_2(a)-f_2(z_2)}
	{2(1+k)^{d_2}},\ldots, \frac{f_n(a)-f_n(z_2)}{2(1+k)^{d_n}}\right)=g(a).$$
	This establishes our claim and hence we obtain
	$$S^d(z_2)\geq \frac{S^d(z_1)-h_{d,\Omega}
		(2f(z_2))}{(1+k)}.$$ Now it follows that
	\begin{align*}
		S^d(z_1)&\leq S^d(z_2)(1+k)+h_{d,\Omega}(2f(z_2))\\
		&=S^d(z_2)+S^d(z_2)k+h_{d,\Omega}(2f(z_2))\\
		&\leq S^d(z_2)+k+B_{2\Omega}\left(\tanh k_{2\Omega}(0,2f(z_2))\right)^{1/L} \ \ \ \ \ \ \ \ \ \ \ \ \ \ \
		\ \ \ \ \ \ \mbox{ (using Lemma \ref{lem:bound})}\\
		&\leq S^d(z_2)+k+B_{2\Omega}\left(\tanh k_{h'(D)}(h'(z_1),h'(z_2))\right)^{1/L}\\
		&= S^d(z_2)+k+B_{2\Omega}(\tanh k_{D}(z_1,z_2))^{1/L}\\
		&=S^d(z_2)+A_{\Omega}\mathcal{K},
	\end{align*}
	where $B_{2\Omega}$ is as in Lemma \ref{lem:bound} for $a=2$, 
	$A_{\Omega}=B_{-\Omega}+B_{2\Omega}$ and $h':D\to 2\Omega$ is defined as $h'(z)=2f(z)$.
	On the similar lines, we can obtain that $$S^d(z_2)\leq S^d(z_1)+A_{\Omega}\mathcal{K}.$$
	Therefore we get 
	\begin{equation}\label{eqn:cts}
		|S^d(z_1)-S^d(z_2)|\leq A_{\Omega}\mathcal{K} \ \mbox{for every}\ z_1, z_2\in D
	\end{equation}
	and hence $S^d$ is continuous.% Here we are using that Kobayashi distance $k_{D}$ is continuous. 
	
\end{proof}
\begin{remark}\label{rem:dhom}
	Let $\Omega_i\subseteq \mathbb{C}^{n_i}$ be $d^i$-balanced $d^i=(d^i_1,d^i_2,\ldots,
	d^i_{n_i}) \in{\mathbb{N}^{n_i}},\ i=1,2,\ldots,k.$	Let $\Omega=\Omega_1\times 
	\Omega_2\times \ldots \times \Omega_k\subseteq \mathbb{C}^n,\ n=n_1+n_2\ldots +n_k$. 
	It is easy to see that $\Omega$ is $d=(d^1,d^2,\ldots,d^k)$-balanced and 
	$\Omega^d(r)=\Omega_1^{d^1}(r_1)\times \Omega_2^{d^2}(r_2)\times\ldots  
	\times \Omega_k^{d^k}(r_k)$(See \cite[~Remark 2.2.14(e)]{pflug}).
\end{remark}

\begin{proposition}\label{prp:dlowerbound} 
	Let $\Omega_i\subseteq \mathbb{C}^{n_i},\ 
	i=1,2,\ldots,k$ be bounded, convex and $d^i$-balanced domains, $d^i\in{\mathbb{N}^{n_i}}$. Let 
	$D_i\subseteq \mathbb{C}^{n_i},\ i=1,2,\ldots,k$ be bounded domains.
	Let $\Omega=\Omega_1\times \Omega_2\times \ldots \times \Omega_k\subseteq \mathbb{C}^n,$ 
	where $n=n_1+n_2+\ldots +n_k$ and $D=D_1\times D_2\times \ldots \times D_k$.
	Let $d=(d^1,d^2,\ldots, d^k)$, then for $a=(a_1,a_2,\ldots,a_k)\in{D}$,  
	\begin{equation}\label{eqn:dlowerbound}
		S^{d}(a)\geq\min_{1\leq i\leq k}  S^{d^i}(a_i). 
	\end{equation}
\end{proposition}
\begin{proof}
	By Theorem \ref{thm:dextremal}, for each $1\leq i\leq k$, there exist an 
	extremal map $f_i$ at $a_i\in{D_i}$. 
	That is, $f_i:D_i\to \Omega_i$ is  injective holomorphic with $f_i(a_i)=0$ such that 
	\begin{equation}\label{eqn:dproduct}
		\Omega_i\left(S^{d^i}(a_i)\right)\subseteq f_i
		(D_i)\subseteq \Omega_i,\ i=1,2,\ldots, k.
	\end{equation}
	Consider the map $f:D \to \Omega$  defined as
	$$f(z_1,z_2,\ldots, z_k):=\left(f_1(z_1),f_2(z_2),\ldots,f_k(z_k)\right).$$
	Clearly, $f$ is injective holomorphic with $f(a)=0$. Let $r=\min_{1\leq i\leq k}  
	S^{d^i}(a_i)$. It follows from Remark \ref{rem:dhom} that 
	$\Omega^d(r)=\Omega_1^{d^1}(r)\times \Omega_2^{d^2}(r)\times\ldots  
	\times\Omega_k^{d^k}(r)$. 
	Let $w=(w_1,w_2,\ldots, w_k) \in \Omega^d(r)=\Omega_1^{d^1}(r)\times \Omega_2^{d^2}(r)
	\times\ldots  \times\Omega_k^{d^k}(r). $ By
	Equation \ref{eqn:dproduct}, there exists $b_i\in{D_i},$ such that $f_i(b_i)=w_i,\ i=1,2,
	\ldots, k,$ since
	$\Omega^{d^i}_{i}(r)\subseteq \Omega^{d{i}}(S^{d^i}(a_i)) $ for each $i$.
	Thus $w=f(b_1,b_2,\ldots, b_k)$ and as $w$ was arbitrarily chosen, we conclude 
	$\Omega^d(r)\subseteq f(D)$. 
	Thus it follows that $S^{d}(a)\geq\min_{1\leq i\leq k} 
	S^{d^i}(a_i) .$
\end{proof}

The following corollary is immediate.
\begin{corollary}
	Product of holomorphic homogeneous $d^i$-regular domains is holomorphic 
	homogeneous $d$-regular, where $d=(d^1,d^2,\ldots,d^k)$.
\end{corollary}

Recall that we say a sequence of subdomains $\{D_n\}$ of $D$ exhausts $D$ if for
each compact subset $K\subseteq D$, there exists $N> 0$ such that $K\subseteq D_k$
for every $k> N$.

\begin{theorem}\label{exhaust}
	If a sequence $D_n\subseteq D$ exhausts $D$, then $\lim_ n S_{D_n}^{d}(z)=
	S^{d}(z)$ uniformly on compact subsets of $D$.
\end{theorem}
\begin{proof}
	This theorem can be proved in a similar manner as in  \cite[~Theorem 3.8]{gen} 
	using Equation \ref{eqn:cts}.
\end{proof}

This theorem\,---using the argument as in \cite[~Theorem 1.2]{deng2019}---\,gives the following
theorem.
\begin{theorem}
	A $d$-balanced domain exhausted by a holomorphic homogeneous $d$-regular domain is 
	holomorphic homogeneous $d$-regular.
\end{theorem}

\section{The squeezing function $S^d$ and the fridman invariant}
In \cite{kaushal}, authors discussed the relation between the squeezing function and 
the Fridman invariant. Similar relation was discussed for between the Fridman invariant and 
the squeezing function corresponding to polydisk,
generalised squeezing function \cite{npoly, gen}. We have the following theorem 
in this direction.
\begin{theorem}\label{thm:equalitypreq}
	Let $D\subseteq \mathbb{C}^n$ be a bounded domain and $\Omega\subseteq \mathbb{C}^n$ be
	bounded, convex, $d$-balanced. Then for $a\in{D}$,
	$$S^{d}(a)^{L}\leq h_{D}^c(a),$$ where $L=\max_{1\leq i\leq n}d_i.$
\end{theorem}
\begin{proof}
	For $a\in D$, let $f:D\to \Omega$ be injective holomorphic map with $f(a)=0$. Let $r>0$ be
	such that $\Omega^d(r)\subseteq f(D).$ Consider $g:\Omega \to D$ defined as 
	$$g(z):=f^{-1}\left(z_1r^{d_1},z_2r^{d_2},\ldots,z_nr^{d_n}\right).$$
	Recall that  $h_{d,\Omega}
	\left(z_1r^{d_1},z_2r^{d_2},\ldots,z_nr^{d_n}\right)=r h_{d,\Omega}(z)$, 
	using Remark \ref{rem:basicminko}(2). Thus for $z\in{\Omega}=\Omega(1),\ \left
	(z_1r^{d_1},z_2r^{d_2},\ldots,z_nr^{d_n}\right)\in{\Omega^d(r)}$ and therefore
	$g$ is well defined. We claim that $B_D^c(a,\tanh^{-1}r^L)\subseteq g(\Omega)\subseteq D.$
	Let $w\in{B^c_{D}\left(a,\tanh^{-1}r^L\right)},$
	then 
	\begin{align*}
		\tanh^{-1}r^L&>c_{D}(a,w)\\
		&=c_{f(D)}(f(a),f(w))\\
		&=c_{f(D)}(0,f(w))\\
		&\geq c_{\Omega}(0,f(w))\\
		&\geq \tanh^{-1}\left(h_{d,\Omega}(f(w))\right)^L.
	\end{align*}  
	We are using  Result \ref{res:bharali} in the last step. This
	gives us $h_{d,\Omega}(f(w))<r$. Thus $w\in{f^{-1}(\Omega^d(r))}$, which upon using
	Remark \ref{rem:basicminko}(2) gives us $w\in g(\Omega)$. Therefore $r^L\leq h_{D}^c(a)$ 
	and hence we get $S^{d}(a)^{L}\leq h_{D}^c(a).$
\end{proof}

\begin{theorem} \label{thm:equalityquotient} 
	If $D, \Omega \subseteq \mathbb{C}^n$ are  bounded, $d$-balanced, convex then 
	$$h_{D}^c(0)^L\leq S^d(0),$$
	where $L=\max_{1\leq i\leq n}d_i.$
\end{theorem}
\begin{proof}
	For $0\in \Omega$, let $f:\Omega\to D$ be an injective holomorphic map with $f(0)=0$. 
	Let $r>0$ be such that $B_D^c(0,r)\subseteq f(\Omega).$ Define $g:D\to \Omega$
	as $$g(w):=f^{-1}(\alpha^{d_1}w_1,\ldots, \alpha^{d_n}w_n ),$$ where $\alpha=\tanh
	r.$ Note that for $w\in D$, $h_{d,D}(w)<1$, which on using 
	Remark \ref{rem:basicminko}(2) gives us $h_{d,D}(\alpha^{d_1}w_1,\ldots, 
	\alpha^{d_n}w_n )<\alpha$
	and therefore $g$ is well defined. Also, $g$ is injective holomorphic with 
	$g(0)=0.$ 
	
	We next claim  that $\Omega^d(\alpha^L)\subseteq g(D).$ To see this, let 
	$z\in \Omega^d(\alpha^L)$, then
	\begin{align*}
		\alpha^L&>h_{d,\Omega}(z)\\
		&\geq \tanh c_{\Omega}(0,z)\\
		&=\tanh c_{f(\Omega)}(f(0),f(z))\\
		&\geq \tanh c_{D}(0,f(z))\\
		&\geq \left(h_{d,D}(f(z))\right)^L.
	\end{align*}
	This yields that $h_{d,D}(f(z))<\alpha,$ and thus we get our claim. This
	further yields $S^d(0)\geq \alpha^L= \left(\tanh r\right)^L,$ 
	which implies that $$h_{D}^c(0)^L\leq S^d(0).$$
\end{proof}
\begin{remark}\label{rem:relate}
	Observe that under the assumption of Theorem \ref{thm:equalityquotient}, using Theorem
	\ref{thm:equalitypreq}  we get
	$$S^d(0)^L \leq h_{D}^c(0)\leq S^d(0)^{1/L}.$$
	In case when $D, \Omega$ are
	bounded, balanced and convex this inequality gives \cite[~Theorem 3]{rong}.
\end{remark}

\begin{theorem}\label{thm:punctured}
	For a bounded, convex, $d$-balanced and homogeneous domain $\Omega$, let $D=\Omega\setminus \{0\}$.
	Then $$S^d(z)^L\leq h_{d,\Omega}(z)\leq  S^d(z)^{1/L}\ \mbox{for all} \ z\in D.$$
\end{theorem}
\begin{proof}
	It follows  directly from the proof of \cite[~Theorem 4.5 ]{gen} and Theorem \ref{thm:equalitypreq}.
	%and \ref{thm:equalityquotient}.
\end{proof}	

\begin{remark}
	Note that when $\Omega$ is bounded, convex, balanced and homogeneous, then Theorem \ref{thm:punctured}
	reduces to the theorem of Rong and Yang \cite[~Theorem 4.5]{gen}, which states that  
	$$\mbox{for}\ z\in D=\Omega\setminus \{0\},\ S_D^{\Omega}(z)=h_{\Omega}(z).$$
\end{remark}

\section*{Acknowledgement}

We are thankful to Gautam Bharali, Peter Pflug and Kaushal Verma for reading our manuscript and 
suggesting many changes. We are also thankful to Fusheng Deng for sending the new arguments for
their proof \cite[Theorem~2.1]{2012}. We profusely thank the referee for comments and suggestions.

\end{document}